\documentclass[final]{dmtcs-episciences}
\usepackage[utf8]{inputenc}
\usepackage{subfigure}

\usepackage[active]{srcltx}
\usepackage{latexsym,amssymb}
\usepackage{stmaryrd}
\usepackage[T1]{fontenc}

\usepackage[pdftex]{hyperref}
\usepackage{mathtools}
\usepackage{tikz}
\usetikzlibrary{snakes,arrows,shapes,external}
\usepackage{bm}
\usepackage[round]{natbib}

\renewcommand{\ge}{\geqslant}
\renewcommand{\le}{\leqslant}
\def\Cl#1{\ensuremath{\mathcal{#1}}}
\def\End#1{\ensuremath{\mathop{\mathrm{End}}#1}}

\newtheorem{Thm}{Theorem}[section]
\newtheorem{Prop}[Thm]{Proposition}
\newtheorem{Lemma}[Thm]{Lemma}
\newtheorem{Cor}[Thm]{Corollary}
\newtheorem{Problem}{Problem}

\title{Binary patterns in the Prouhet-Thue-Morse sequence}

\author{Jorge Almeida\affiliationmark{1}\thanks{%
    Work partially support by CMUP (UID/MAT/ 00144/2019), which is
    funded by FCT (Portugal) with national (MATT'S) and European
    structural funds (FEDER) under the partnership agreement PT2020. %
    The work was carried out at Masaryk University, whose hospitality
    is gratefully acknowledged, with the support of the FCT sabbatical
    scholarship SFRH/BSAB/142872/2018.} %
  \and Ond\v rej Kl\'ima\affiliationmark{2}\thanks{%
    Work supported by Grant 19-12790S of the Grant Agency of the Czech
    Republic.}}

\affiliation{%
  CMUP, Dep.\ Matem\'atica, Faculdade de Ci\^encias, Universidade do
  Porto, Portugal\\
  Dept.\ of Mathematics and Statistics, Masaryk University, Brno,
  Czech Republic}

\received{2019-05-15}
\revised{2021-07-27}
\accepted{2021-08-02}


\begin{document}

\publicationdetails{23}{2021}{3}{6}{5460}
\maketitle

\begin{abstract}
  We show that, with the exception of the words $a^2ba^2$ and
  $b^2ab^2$, all (finite or infinite) binary patterns in the
  Prouhet-Thue-Morse sequence can actually be found in that sequence
  as segments (up to exchange of letters in the infinite case). This
  result was previously attributed to unpublished work by D. Guaiana and
  may also be derived from publications of A. Shur only available in
  Russian. We also identify the (finitely many) finite binary patterns
  that appear non trivially, in the sense that they are obtained by
  applying an endomorphism that does not map the set of all segments
  of the sequence into itself.
\end{abstract}

\section{Introduction}
\label{sec:intro}

Let $\mu$ be the endomorphism of the free semigroup $\{a,b\}^+$
defined by $\mu(a)=ab$ and $\mu(b)=ba$. Since $a$ is a prefix of
$\mu(a)$, $\mu^n(a)$ is also a prefix of~$\mu^{n+1}(a)$. Hence, the
sequence $(\mu^n(a))_n$ determines a sequence of letters, or infinite
word, whose prefix of length $2^n$ is $\mu^n(a)$; we say that the
infinite word $\mathbf{t}$ thus obtained is \emph{generated} by~$\mu$.
It is called the \emph{Prouhet-Thue-Morse sequence} and it has been
the object of extensive studies and applications. It was first
considered by \citet{Prouhet:1851} in connection with a problem in
number theory, five decades later by \citet{Thue:1906,Thue:1912bis} to
exhibit infinite words avoiding cubes and squares, and another two
decades later by \citet{Morse:1921a} as a discretized description of
non-periodic recurrent geodesics in surfaces of negative curvature.
See \citet{Allouche&Shallit:1999} for a survey on this topic,
including several further connections with other branches of
Mathematics. The first author and other collaborators have previously
studied the sequence $\mathbf{t}$ in the framework of symbolic
dynamics and its connections with free profinite semigroups
(see~\citet{Almeida&ACosta:2013}
and~\citet{Almeida&ACosta&Kyriakoglou&Perrin:2020b}). It was in fact
an attempt to construct a profinite semigroup with certain properties
that prompted this work, although no further references to profinite
semigroups will be made in this paper.

This paper concerns the study of binary patterns of~$\mathbf{t}$, that
is, finite or infinite words $w$ over the alphabet $\{a,b\}$ for which
there exists an endomorphism $\varphi$ of the semigroup $\{a,b\}^+$
(naturally extended to infinite words) such that the word $\varphi(w)$
can be found as a block of consecutive letters of~$\mathbf{t}$ (which
we call a \emph{segment} of~$\mathbf{t}$). Since we need to identify
concrete finite segments of~$\mathbf{t}$, a simple and efficient
algorithm on how to compute them is presented in
Section~\ref{sec:segments}.

Characterizations of binary patterns of~$\mathbf{t}$ are due to
\citet{Shur:1996a} and D. Guaiana (unpublished work announced
in~\citet{Restivo&Salemi:2002b,Restivo&Salemi:2002a}). Our first main
contribution is a proof of the characterization attributed to D.
Guaiana (but also, independently, obtained by \citet{Shur:1997PhD} in
his thesis) using results from~\citet{Shur:1996a}: with the
exception of $a^2ba^2$ and $b^2ab^2$, the binary patterns
of~$\mathbf{t}$ are its finite segments. Section~\ref{sec:finite}
presents our proof of this result.

The endomorphism $\mu$ and the endomorphism $\xi$ exchanging the
letters $a$ and $b$ are easily seen to transform finite segments
of~$\mathbf{t}$ into other such segments. In
Section~\ref{sec:typical}, we consider the problem of determining
which finite segments may only be transformed into other segments by
endomorphisms of~$\{a,b\}^+$ that may be obtained by composition of
$\mu$ and $\xi$. Such words are said to be \emph{typical} since we
show that all but finitely many finite segments of~$\mathbf{t}$ are
typical. We further determine all atypical words. As an application of
our results, we also determine all infinite binary patterns
of~$\mathbf{t}$.

We conclude the paper with Section~\ref{sec:problems}, where we
propose the investigation of the properties we established
for~$\mathbf{t}$ for arbitrary infinite words.

\section{Segments of \texorpdfstring{$\mathbf{t}$}{t}}
\label{sec:segments}

By a \emph{word} we always mean a finite sequence of letters of an
alphabet $A$, that is a member of the free monoid $A^*$. A word $u$ is
a \emph{factor} of a word $v$ if there exist words $x$ and $y$ such
that $v=xuy$. In spite of the terminology, an \emph{infinite word} is
not a word but rather an infinite sequence of letters.

Note that $\{ab,ba\}$ is a code, in the sense that it generates a free
subsemigroup of~$\{a,b\}^+$ and, therefore, $\mu$ is injective.

For an infinite word $w=a_1a_2\cdots $, by the \emph{segments} of $w$
we mean the words of the form $a_ka_{k+1}\cdots a_\ell$ with $k\le \ell$ and
the infinite words of the form $a_ka_{k+1}\cdots$. Note that, since
$\mu^{n+1}(a)=\mu^n(a)\mu^n(b)$, all factors of the words $\mu^n(b)$
are segments of~$\mathbf{t}$. It follows that a word $u\in\{a,b\}^+$
is a segment of~$\mathbf{t}$ if and only if so is the word that is
obtained from~$u$ by interchanging the letters $a$ and $b$.

A word $w\in A^+$ is said to be \emph{avoided} by~$\mathbf{t}$ if
there is no homomorphism $\varphi:A^+\to\{a,b\}^+$ such that
$\varphi(w)$ is a segment of~$\mathbf{t}$. We also say that $w\in A^+$
is \emph{unavoidable} in~$\mathbf{t}$ if it is not avoided
by~$\mathbf{t}$; we then also say that $w$ is a \emph{pattern}
of~$\mathbf{t}$. For instance, it is well known that $a^3$ and $ababa$
are avoided by~$\mathbf{t}$, which is also expressed by saying that
$\mathbf{t}$ is, respectively, cube-free and overlap-free
\citep{Lothaire:1983}.

The preceding notions are extended to infinite words by saying how
endomorphisms of~$\{a,b\}^+$ are applied to infinite words. Given an
infinite word $w=a_1a_2\cdots$ over the alphabet $\{a,b\}$ and an
endomorphism $\varphi$ of~$\{a,b\}^+$, we let $\varphi(w)$ be infinite
word obtained by concatenating the $\varphi(a_i)$:
$\varphi(w)=\varphi(a_1)\varphi(a_2)\cdots$.

For a nonempty word~$u$, let $t_1(u)$ denote its last letter.

The computation of the segments of~$\mathbf{t}$ may be carried out
easily in view of the following proposition. The first part is an
improved version of~\citet[Corollary~1]{Shur:1996b}, although the same
conclusion is in fact already established in the proof of the cited
statement. We present a proof for the sake of completeness.

\begin{Prop}
  \label{p:segments}
  A word $w$ is a segment of~$\mathbf{t}$ if and only if it is a
  factor of $\mu^n(a)$, where $n=1$ if $|w|=1$, $n=3$ if $|w|=2$, and
  $n=2+\lceil\log_2(|w|-1)\rceil$ otherwise. Moreover, for every
  integer $k\ge3$, the value $n=2+\lceil\log_2(k-1)\rceil$ is minimum
  for $\mu^n(a)$ to admit as factors all segments of~$\mathbf{t}$ of
  length~$k$.
\end{Prop}

\begin{proof}
  Since $\mathbf{t}$ is cube-free, the cases where $|w|\le3$ are
  easily verified by inspection. Suppose that $w$ is a segment
  of~$\mathbf{t}$ which we may assume to be of length at least~4.
  Then, $w$~is a factor of $\mu^n(a)$ for some positive integer~$n$.
  Take $n$ to be minimum with that property. If $m$ is the minimum
  positive integer such that $w$ is a factor of~$\mu^m(x)$ for some
  letter~$x$, then either only $x=b$ can play that role and $n=m+1$,
  or $x=a$ may play it and $n=m$. We need to show, respectively, that
  $m\le 1+\lceil\log_2(|w|-1)\rceil$ or %
  $m\le 2+\lceil\log_2(|w|-1)\rceil$. Since $|w|\ge4$, we may assume
  that $m\ge4$ for, otherwise, the inequality $m\le
  1+\lceil\log_2(|w|-1)\rceil$ holds trivially.

  Let $\mu(x)=xy$. As $\mu^m(x)=\mu^{m-1}(xy)$ and $m$ is minimum,
  there must be a nontrivial factorization $w=w_1w_2$ with $w_1$ a
  suffix of~$\mu^{m-1}(x)$ and $w_2$ a prefix of~$\mu^{m-1}(y)$.

  If one of the factors $w_1$ or $w_2$ has length greater than
  $2^{m-2}$, then we must have $|w|>2^{m-2}+1$, which implies that %
  $m\le 1+\lceil\log_2(|w|-1)\rceil$ and fulfills our aim. Thus, we
  may assume that both $w_1$ and $w_2$ have length at most $2^{m-2}$.
  Now, we have $\mu^m(x)=\mu^{m-2}(xyyx)$ and $w$ is a factor of
  $\mu^{m-2}(yy)=\mu^{m-3}(yxyx)$. If $w$ is a factor of either
  $\mu^{m-3}(xyx)$ or $\mu^{m-3}(yxy)$, then it is also a factor of
  $\mu^{m-2}(xx)=\mu^{m-3}(xyxy)$ and, therefore, also of $\mu^m(y)$,
  so that we are in the case $n=m$. On the other hand, by the
  minimality of~$m$ the word $w$ cannot be a factor of
  $\mu^{m-3}(xy)=\mu^{m-2}(x)$ or $\mu^{m-3}(yx)=\mu^{m-2}(y)$, and so
  we have $|w|\ge 2^{m-3}+2$, which implies that $n=m\le
  2+\lceil\log_2(|w|-1)\rceil$, as claimed. It remains to consider the
  case where $w$~is a factor of neither $\mu^{m-3}(xyx)$ nor
  $\mu^{m-3}(yxy)$. Then there must be a factorization
  $w=s\mu^{m-3}(xy)z$ with $s$ and $z$ nontrivial words, so that
  $|w|\ge2^{m-2}+2$, which yields $m\le 1+\lceil\log_2(|w|-1)\rceil$.
  This completes the proof of the first part of the proposition.

  To prove the last part of the proposition, first note that, for
  $k\ge3$, the value of $f(k)=2+\lceil\log_2(k-1)\rceil$ is at
  least~3. We claim that, for $n\ge3$, there is a word of length
  $2^{n-3}+2$ that is a segment of~$\mathbf{t}$ but not a factor
  of~$\mu^{n-1}(a)$. Noting that $f(2^{n-3}+2)=n$, the result follows.

  To establish the claim, consider the word
  $w=t_1(\mu^{n-3}(b))\mu^{n-3}(a)b$. It is a segment of~$\mathbf{t}$,
  in fact a factor of
  $\mu^n(a)=\mu^{n-2}(a)\mu^{n-3}(b)\cdot\mu^{n-3}(a)\cdot\mu^{n-1}(b)$
  since $b$ is the first letter of~$\mu^{n-1}(b)$. It remains to show
  that $w$ is not a factor
  of~$\mu^{n-1}(a)=\mu^{n-3}(a)\mu^{n-3}(b)\mu^{n-3}(b)\mu^{n-3}(a)$.
  Otherwise, since there are no overlaps in~$\mathbf{t}$, $w$ must be
  a factor of $\mu^{n-3}(bb)$. For $n=3$, this is clearly impossible
  since not even $\mu^{n-3}(a)=a$ is a factor of $\mu^{n-3}(bb)=bb$.
  For $n>3$, we have
  \begin{align}
    \label{eq:segments-a}
    \mu^{n-3}(a)&=\mu^{n-4}(a)\mu^{n-4}(b)\\
    \label{eq:segments-b}
    \mu^{n-3}(bb)&=\mu^{n-4}(b)\mu^{n-4}(a)\mu^{n-4}(b)\mu^{n-4}(a).
  \end{align}
  Since there are no overlaps in $\mu^{n-3}(bb)$, the only place where
  $\mu^{n-3}(a)$ is found as a factor of $\mu^{n-3}(bb)$ is as the
  product of the two middle factors in the factorization given
  in~\eqref{eq:segments-b}. Hence, the word
  $w=t_1(\mu^{n-3}(b))\mu^{n-3}(a)b$ is not a factor
  of~$\mu^{n-3}(bb)$ since, for instance, $b$ is not the first letter
  of $\mu^{n-4}(a)$.
\end{proof}

For example, the segments of lengths 4 and~5 of the infinite
word~$\mathbf{t}$ are the factors of those lengths of
$\mu^4(a)=abbabaabbaababba$. But, for instance, $aabb$ is a segment
of~$\mathbf{t}$ but not a factor of~$\mu^3(a)=abbabaab$; it is
precisely the segment considered in the last part of the proof of
Proposition~\ref{p:segments}.

Throughout the remainder of the paper, when we need to check whether a
concrete finite word is a segment of~$\mathbf{t}$, without any further
reference we simply apply the algorithm given by
Proposition~\ref{p:segments}, which is linear in the length of the
given word. We proceed similarly when we need to compute all the
segments of~$\mathbf{t}$ of a given length.

Now, we take into account also the dual version of
Proposition~\ref{p:segments}, where $\mu^n(b)$ is considered instead of
$\mu^n(a)$, which is a direct consequence of
Proposition~\ref{p:segments} using the fact that the set of all
segments of~$\mathbf{t}$ of fixed length~$k$ is closed under taking
images under $\xi$.
Since every segment of $\mathbf{t}$ of length $2^{n+1}-1$ must contain
the factor $\mu^n(a)$ or $\mu^n(b)$, it follows from
Proposition~\ref{p:segments} that, for $k\ge3$, every segment of
length $k$ of~$\mathbf{t}$ is a factor of every segment
of~$\mathbf{t}$ of some length $\ell$ which is at most
\begin{displaymath}
  2^{2+\lceil\log_2(k-1)\rceil+1}-1\le 2^{4+\log_2(k-1)}-1=16k-17.
\end{displaymath}
The existence of such an $\ell$ is the property known as \emph{uniform
  recurrence} of~$\mathbf{t}$ and holds for every sequence generated
by iterating a primitive endomorphism of a free semigroup
\citep[Proposition~5.2]{Queffelec:2010}. In the case of the
Prouhet-Thue-Morse sequence, the optimum value of~$\ell$ is presented
in~\citet[Example~10.9.3]{Allouche&Shallit:2003}: for $k\ge3$, we have
$\ell=9\cdot2^r+k-1$, where $r$ is the integer determined by the
inequalities $2^r+2\le k\le2^{r+1}+1$. Note that using the first
inequality determining $r$, one gets the upper bound $\ell\le10k-19$,
which is better than our rough upper bound $\ell\le16k-17$.

\section{Finite binary patterns}
\label{sec:finite}

The following result plays a key role below.

\begin{Thm}[\citet{Shur:1996a}]
  \label{t:Shur}
  The set of words of $\{a,b\}^+$ that are avoided by $\mathbf{t}$ is
  the fully invariant ideal generated by the set
  \begin{align*}
    &\{a^3,ababa,a^2ba^2b,ab^2ab^2, %
    t_1(\mu^k(a))\mu^k(aba)a\ (k\ge1), 
    t_1(\mu^m(a))\mu^m(bab)a\ (m\ge2)\}.
  \end{align*}
  Moreover, the above is a minimal generating set for the fully
  invariant ideal of the words avoided by~$\mathbf{t}$.
\end{Thm}

The generators corresponding to $k=1$ and $k=2$ are, respectively,
$bab^2a^2ba$ and $a\,ab^2a\,ba^2b\,ab^2a\,a$; the generator
corresponding to $m=2$ is $a\,ba^2b\,ab^2a\,ba^2b\,a$ while, for
$m=1$, the word given by $t_1(\mu^m(a))\mu^m(bab)a=b^2a^2b^2a^2$ is
avoided by~$\mathbf{t}$ but may be obtained for instance from the
generator $a^2ba^2b$ by mapping $a$ to~$b$ and $b$ to~$a^2$.

Another useful ingredient in our arguments is the following
``synchronization'' result.

\begin{Lemma}[{\citet[Lemma~3.9]{Luca&Varricchio:1989}}]
  \label{l:Luca-Varricchio}
  Let $X=\{ab,ba\}$ and consider $s\in X^+$ with $|s|\ge4$. If $u$ and
  $v$ are words such that $usv\in X^+$ and $|u|$ is odd then $usv$ has
  an overlap.
\end{Lemma}

Since $\mathbf{t}$ has no overlaps, we conclude that $\mathbf{t}=usv$,
with $s$ as in Lemma~\ref{l:Luca-Varricchio} and $u$ a finite word,
then $u$ has even length.

\begin{Cor}
  \label{c:Luca-Varricchio+}
  If there is a factorization $\mathbf{t}=u\mu^{n+1}(x)v$ where
  $u\in\{a,b\}^*$, $x\in\{a,b\}$, and $n\ge0$, then $u=\mu^n(w)$ and
  $v=\mu^n(z)$ for some word $w$ and infinite word $z$.
\end{Cor}

\begin{proof}
  We proceed by induction on $n$, the case $n=0$ being trivial as
  $\mu^0$ is interpreted to be the identity function. Suppose that
  $n\ge1$ and, by symmetry, assume that $x=a$. Since
  $\mu^{n+1}(a)=\mu^n(a)\mu^n(b)=\mu^{n-1}(abba)$, by the induction
  hypothesis we know that $u=\mu^{n-1}(w)$ and $v=\mu^{n-1}(z)$, for
  some finite word $w$ and infinite word $z$, and so
  $\mathbf{t}=wabbaz$. Lemma~\ref{l:Luca-Varricchio} then implies that
  $w$ and $z$ belong to the image of~$\mu$. Hence, $u$ and $v$ belong
  to the image of~$\mu^n$.
\end{proof}

We say that a segment $u$ of~$\mathbf{t}$ is \emph{special} if both
$ua$ and $ub$ are segments of~$\mathbf{t}$. The special segments
of~$\mathbf{t}$ have been investigated by \citet{Luca&Varricchio:1989}
with the purpose of counting the number of segments of each given
length. For our purposes, it suffices to observe the following much
simpler result.

\begin{Lemma}[{\citet[Lemma~3.6]{Luca&Varricchio:1989}}]
  \label{l:special}
  If the word $w$ is special, then so is $\mu(w)$.
\end{Lemma}

We say that two words are \emph{suffix comparable} if at least one of
them is a suffix of the other. The following lemma is the core of
our arguments.

\begin{Lemma}
  \label{l:extension-a}
  Suppose that $u$ is a finite word such that $ua$ is unavoidable
  in~$\mathbf{t}$ and $ub$ is a segment of~$\mathbf{t}$ but $ua$ is
  not. If $n\ge2$ and $u$ is suffix comparable with $\mu^n(a)$, then
  $u$ is also suffix comparable with $\mu^{n+1}(b)$.
\end{Lemma}

\begin{proof}
  Since $\mu^n(a)$ is a suffix of~$\mu^{n+1}(b)=\mu^n(b)\mu^n(a)$, we
  may assume that $\mu^n(a)$ is a suffix of~$u$, say $u=u_1\mu^n(a)$.
  Consider a concrete occurrence of $ub$ in~$\mathbf{t}$:
  $\mathbf{t}=u_0ubv$, where $u_0\in\{a,b\}^*$. Since $\mathbf{t}$ is
  recurrent, we may assume that $|u_0|\ge2^n$. By
  Corollary~\ref{c:Luca-Varricchio+}, we know that
  $u_0u_1=\mu^{n-1}(u')$ and $bv=\mu^{n-1}(v')$ for some word $u'$ and
  infinite word $v'$ (cf.~Figure~\ref{fig:segs-t}). Since
  $\mu^{n-1}(v')$ starts with $b$, so does $v'$.
  \begin{figure}[ht]
    \centering
    \unitlength=0.5mm
    \begin{picture}(200,30)
      \thinlines
      \multiput(0,30)(0,-10){3}{\line(1,0){200}}
      \multiput(0,0)(0,20){2}{\line(0,1){10}}
      \put(40,30){\line(0,-1){20}}
      \put(100,30){\line(0,-1){30}}
      \put(60,20){\line(0,-1){20}}
      \put(0,0){\line(1,0){60}}
      \put(100,0){\line(1,0){100}}
      \put(17,23){$u_0$}
      \put(68,23){$u$}
      \put(128,23){$bv$}
      \put(47,13){$u_1$}
      \put(71,13){$\mu^n(a)$}
      \put(120,3){$\mu^{n-1}(v')$}
      \put(17,3){$\mu^{n-1}(u')$}
    \end{picture}
    \caption{Some segments of $\mathbf{t}$}
    \label{fig:segs-t}
  \end{figure}

  Suppose first that $u'$ ends with the letter $b$ and
  $|u|>3\cdot2^{n-1}$. If $u'$ ends in $ab$ then the word
  $t_1(\mu^{n-1}(a))\mu^{n-1}(bab)a$ is a suffix of~$ua$ which, in
  view of Theorem~\ref{t:Shur}, contradicts the assumption that $ua$
  is unavoidable in~$\mathbf{t}$. On the other hand, if $u'$ ends with
  $b^2$ then, taking into account that $bv$ starts with
  $\mu^{n-1}(b)$, we conclude that
  $\mathbf{t}=\mu^{n-1}(u''b^2ab^2v'')$ for some finite word $u''$ and
  infinite word $v''$. Since $\mathbf{t}$ is a fixed point of the
  injective endomorphism~$\mu$, it follows that $b^2ab^2$ is a segment
  of~$\mathbf{t}$, which we know is not the case.

  If $|u|\le 3\cdot2^{n-1}$, then $u$ is suffix of~$\mu^{n-1}(xab)$
  for a letter $x$. By Lemma~\ref{l:special}, as it is easy to check
  that $xab$ is special, so is $\mu^{n-1}(xab)$. Hence, $u$ is
  special, contradicting the assumption that $ua$ is not a segment
  of~$\mathbf{t}$.

  Thus, $u'$ must end with $ba$, so that $u_0u$ ends
  with~$\mu^{n+1}(b)$. Since both $u$ and $\mu^{n+1}(b)$ are suffixes
  of~$u_0u$, they must be suffix comparable, thereby concluding the
  proof of the lemma.
\end{proof}

Similarly, one can prove the following lemma.

\begin{Lemma}
  \label{l:extension-b}
  Suppose that $u$ is a finite word such that $ua$ is unavoidable
  in~$\mathbf{t}$ and $ub$ is a segment of~$\mathbf{t}$ but $ua$ is
  not. If $n\ge2$ and $u$ is suffix comparable with $\mu^n(b)$ then
  $u$ is also suffix comparable with $\mu^{n+1}(a)$.
\end{Lemma}

Shur also observed in~\citet{Shur:1996a} that the word $a^2ba^2$ (and,
therefore, also $b^2ab^2$) is unavoidable in~$\mathbf{t}$ but it is
not a segment of~$\mathbf{t}$. Our first main result is that there are
no other such examples, thus providing an alternative characterization
of two-letter words unavoidable in~\textbf{t}.

According to~\citet[Theorem~2]{Restivo&Salemi:2002a}
and~\citet[Theorem~3]{Restivo&Salemi:2002b}, the following theorem,
which is considered to be surprising, was first proved by D. Guaiana
in~1996 but, through private communication with A. Restivo, we learned
that the proof was never published and the manuscript appears to be
lost. On the other hand, we later learned from A. M. Shur that the
next theorem also appears in his Ph.D. thesis (\citeyear{Shur:1997PhD}),
which has never been published other than as a document in the Russian
State Library. Moreover, Shur observed that the result can also easily
be drawn from Theorem~\ref{t:Shur} using a characterization of the
finite words on the alphabet $\{a,b\}$ that are not segments
of~$\mathbf{t}$, which is given in~\citet[Statement~1]{Shur:2005}, a
paper also in Russian. Since all the proofs seem to be either lost or
somewhat inaccessible in the Russian literature, again we present our
own proof for the sake of completeness.

\begin{Thm}
  \label{t:unavoidable-in-TM}
  A word $w\in\{a,b\}^+$ is unavoidable in~$\mathbf{t}$ if and only if
  it is one of the words $a^2ba^2$ and $b^2ab^2$, or it is a segment
  of~$\mathbf{t}$.
\end{Thm}

\begin{proof}
  We proceed by induction on the length of~$w$. In view of
  Theorem~\ref{t:Shur}, it is easy to check that the theorem holds for
  words of length at most~5. Assuming inductively that the result
  holds for words of length $n$, let $w$ be a word of length $n+1\ge6$
  that is unavoidable in~$\mathbf{t}$. Since interchanging the letters
  $a$ and $b$ does not affect either of the properties of being
  unavoidable in~$\mathbf{t}$ and being a segment of~$\mathbf{t}$, we
  may as well assume that $a$~is the last letter of~$w$.

  Let $w=ua$. Since $w$ is unavoidable in~$\mathbf{t}$, so is $u$.
  Hence, by induction hypothesis, $u$ may be found somewhere as a
  segment of~$\mathbf{t}$. Take such an occurrence of~$u$
  in~$\mathbf{t}$ and let $x$ be the letter immediately after it. We
  wish to show that there is such an occurrence of~$u$ in~$\mathbf{t}$
  with $x=a$. Aiming at a contradiction, we may assume that there is
  no such occurrence, that is, we always have $x=b$. Since
  $\mathbf{t}$ is recurrent, the segment $ux$ may be found
  in~$\mathbf{t}$ as far as we wish, so that we may continue
  prolonging it on the left as much as may be convenient. Thus, we are
  assuming that $ua$ is unavoidable in~$\mathbf{t}$ and that $ub$ is a
  segment of~$\mathbf{t}$ as long as desired but $ua$~is not a segment
  of~$\mathbf{t}$. However, we have to be careful because there is in
  principle no assurance that such an extension of $u$ to the left
  retains the property that $ua$ is unavoidable in~$\mathbf{t}$.

  Since $a^3$ is avoided by~$\mathbf{t}$ and we are assuming that
  $x=b$, $u$ cannot end with $b^2$. We distinguish several cases
  according to the termination of the word $u$.

  If $u$ ends with $b$ then, by the above, it ends with $ab$. Suppose,
  more precisely, that $u$ ends with $bab$. Since $w$ ends with $baba$
  and $ababa$ is avoided by~$\mathbf{t}$, in fact $u$ must end with
  $b^2ab$. This situation is impossible since we know that the suffix
  $b^2ab^2$ of~$ub$ is not a segment of~$\mathbf{t}$.

  Alternatively, assuming that $u$ ends with $b$, it must end with
  $ba^2b=\mu^2(b)$. We may then apply successively
  Lemmas~\ref{l:extension-a} and~\ref{l:extension-b} to deduce that
  there is some $n\ge2$ such that $u$ is a suffix of $\mu^n(a)$. By
  Lemma~\ref{l:special}, it follows that $u$ is special, which
  contradicts the assumption that $ua$ is not a segment
  of~$\mathbf{t}$.

  The next case we consider is that where $u$ ends with $aba$. Note
  that $u$ cannot end with $a^2ba$ for, otherwise, $w=ua$ ends with
  $ba^2ba^2$ and, therefore, it cannot be unavoidable in~$\mathbf{t}$.
  Also, $u$ cannot end with $baba$ since $babab$ is not a segment
  of~$\mathbf{t}$. Hence, $aba$ is not a suffix of~$u$.

  Thus, assuming that $u$ ends with $ba$, it must end with
  $ab^2a=\mu^2(a)$. We are then again led to a contradiction as above
  using Lemmas~\ref{l:extension-a}, \ref{l:extension-b}
  and~\ref{l:special}.
\end{proof}

\section{Typical finite binary patterns}
\label{sec:typical}

Recall that the endomorphism of~$\{a,b\}^+$ switching the letters $a$
and $b$ is denoted $\xi$. Since the finite segments of~$\mathbf{t}$
are the finite words that are factors of~$\mu^n(a)$ for all
sufficiently large $n$ and $\mu^n(a)$ is a factor of $\mu^{n+1}(b)$,
we may replace $a$ by $b$ in that characterization of the segments
of~$\mathbf{t}$. Moreover, as $\xi$ commutes with~$\mu$, we conclude
that the set of finite segments of~$\mathbf{t}$ is closed under
applying the substitutions $\xi$ and~$\mu$. By induction on~$n$, the
words $\mu^{2n}(a)$ are palindromic, in the sense that they coincide
with the words read in the reverse order; this entails the well known
fact that the set of segments of~$\mathbf{t}$ is closed under
reversal.

We say that a word $w\in\{a,b\}^+$ is \emph{atypical} if it is a
segment of~$\mathbf{t}$ and there is an endomorphism $\varphi$
of~$\{a,b\}^+$ such that $\varphi(w)$ is also a segment
of~$\mathbf{t}$ and $\varphi$ is not of one of the forms $\mu^n$ or
$\xi\circ\mu^n$ with $n\ge0$. Segments of~$\mathbf{t}$ that are not
atypical are said to be \emph{typical}.

We say that a word is a \emph{variant} of another word $w$ if it may be
obtained from $w$ by applying reversal or~$\xi$ or both. Note that the
set of atypical words is closed under taking factors and, by the above
discussion, it is also closed under taking variants.

The following result appears explicitly
as~\citet[Proposition~3.3]{Brlek:1989a} but may already be extracted
from~\citet{Thue:1912bis} (see~\citet[Chapter~3,
Proposition~2.13]{Berstel:1995}).

\begin{Prop}
  \label{p:Brlek}
  If $u^2$ is a segment of~$\mathbf{t}$ then $u$ is one of the words
  $\mu^n(a)$, $\mu^n(b)$, $\mu^n(aba)$ or $\mu^n(bab)$ for some $n\ge0$.
\end{Prop}

Yet another property of the Prouhet-Thue-Morse infinite word is the
following result which explains the above terminology.

\begin{Thm}
  \label{t:substitutions}
  Let $w\in\{a,b\}^+$ be a segment of~$\mathbf{t}$ containing at least
  one of the segments $aba$ and $bab$ along with all other segments
  of~$\mathbf{t}$ of length~3. Then $w$ is typical.
\end{Thm}

\begin{proof}
  Suppose $\varphi$ is an endomorphism of~$\{a,b\}^+$ such that
  $\varphi(w)$ is a segment of~$\mathbf{t}$. By
  Proposition~\ref{p:Brlek}, since $\varphi(a^2)$ and $\varphi(b^2)$
  are square segments of~$\mathbf{t}$, each of the words $\varphi(a)$
  and $\varphi(b)$ must be obtained by applying a power of~$\mu$ to
  one of the words $a,b,aba,bab$. Let then $\varphi(a)=\mu^k(u)$ and
  $\varphi(b)=\mu^\ell(v)$, where $u,v\in\{a,b,aba,bab\}$. We may
  assume that $k\le\ell$ since, otherwise, we would consider the pair
  $(\xi(w),\varphi\circ\xi)$ instead of $(w,\varphi)$. Then, we have the
  factorization $\varphi=\mu^k\circ\psi$, where $\psi$ is the endomorphism
  of~$\{a,b\}^+$ defined by $\psi(a)=u$ and $\psi(b)=\mu^{\ell-k}(v)$.
  Since $\mu$ is injective and $\mathbf{t}$ is a fixed point of~$\mu$,
  from the fact that $\varphi(w)$ is a segment of~$\mathbf{t}$, we
  conclude that so is $\psi(w)$. On the other hand, since $\xi$ and
  $\mu$ commute, if $\psi$ is a product of $\mu$ and~$\xi$ then so is
  $\varphi$. Hence, we may assume that $k=0$.

  The mapping $\xi\circ\varphi$ has also the property that
  $(\xi\circ\varphi)(w)$ is a segment of~$\mathbf{t}$. Since
  $(\xi\circ\varphi)(a)=\xi(\mu^k(u))=\mu^k(\xi(u))$, we may further
  assume that $u$ is one of the words $a$ or $aba$. Since $aab$ is a
  factor of~$w$ and neither $a^3$ nor $a^2ba^2$ is a factor
  of~$\mathbf{t}$, then $v$ cannot start with $a$ and, therefore, it
  must be either $b$ or $bab$.

  Consider first the case where $u=a$. Since $baa$ is a factor of~$w$
  but $a^3$ is not a factor of~$\mathbf{t}$, $\varphi(b)$ cannot end
  in $a$, which implies that $\ell$ is even. If $\ell\ge2$, then
  $\varphi(b)$ starts with $\mu^2(b)=ba^2b$. But, since $a^2b$ is a
  factor of~$w$, this implies $a^2ba^2$ is a factor of $\varphi(w)$
  and, therefore, a segment of~$\mathbf{t}$, which we know is not the
  case. Hence, we must have $\ell=0$. It remains to rule out the case
  $\varphi(b)=bab$, which results from noting that in that case, from
  the assumption that either $aba$ or $bab$ is a factor of~$w$, it
  follows that either $ababa$ or $bababab$ is a factor of~$\varphi(w)$
  while we know that $ababa$ is not a segment of~$\mathbf{t}$.

  Next, consider the case where $u=aba$. Since $a^2b$ is a factor
  of~$w$ and $ababa$ is not a segment of~$\mathbf{t}$, $\varphi(b)$
  cannot start with $ba$. As $\mu^\ell(b)$ is a prefix
  of~$\varphi(b)$, it follows that $\ell=0$ and $\varphi(b)=b$. This
  leads to a similar situation as that considered at the end of the
  preceding paragraph, with the letters $a$ and $b$ interchanged which
  is, therefore, excluded. This completes the proof of the theorem.
\end{proof}

The assumption of Theorem~\ref{t:substitutions} that a segment
of~$\mathbf{t}$ contains as factors at least one of the words $aba$
and $bab$ along with all other segments of length 3 of~$\mathbf{t}$
holds for all segments of~$\mathbf{t}$ of length $10$, as may be
easily checked by examining all segments of that length. Hence, by
Theorem~\ref{t:substitutions} there are only finitely many atypical
words. Since there are 5 different segments of length~3
of~$\mathbf{t}$ that are supposed to appear in the word, no word with
length shorter than 7 satisfies the criterion of
Theorem~\ref{t:substitutions} and $ab^2a^2ba$ is a word of length~7
that does satisfy it. On the other hand, the segment $a^2bab^2aba$, of
length~9, fails to have the segment $ba^2$ as a factor.

The following result completes the above observations by giving the
full identification of atypical words.

\begin{Thm}
  \label{t:atypical}
  Up to taking variants, the atypical words are the factors of the
  words $aabab$, $abaaba$, and $aabbaab$.
\end{Thm}

\begin{proof}
  To check that all relevant words have been duly considered, the reader
  may wish to refer to the diagram in Figure~\ref{fig:eggbox} later in
  the paper, where all atypical words are represented.
  
  The following is the complete list of segments of~$\mathbf{t}$ of
  length~5:
  \begin{align*}
    &aabab, aabba, abaab, ababb, abbaa, abbab, \\
    &baaba, baabb, babaa, babba, bbaab, bbaba.
  \end{align*}
  Note that all these words are variants of factors of at least one of
  the three words in the statement of the theorem. Hence, by showing that
  those three words are atypical, we obtain that so are all words of
  length up to~5.
  
  We next indicate for each of the words in the statement of the
  theorem an endomorphism $\varphi$ of~$\{a,b\}^+$ not of the forms
  $\mu^n$ and $\xi\circ\mu^n$ that maps it to a segment
  of~$\mathbf{t}$:
  \begin{itemize}
  \item $aabab$: $\varphi(a)=a$, $\varphi(b)=b^2aba^2b$;
  \item $abaaba$: $\varphi(a)=a$, $\varphi(b)=b^2$;
  \item $aabbaab$: $\varphi(a)=a$, $\varphi(b)=bab$.
  \end{itemize}
  The verification of all these statements amounts to routine
  calculations.

  Showing that there are no other atypical words requires
  more work. Note that a word is typical if it has a typical factor.
  Hence, also excluding variants and words that satisfy the criterion
  of Theorem~\ref{t:substitutions}, we obtain the following reduced
  list of words remaining to be treated:
  \begin{equation}
    \label{eq:small-typical}
    aababb, aabbab, abbaabba, ababba.
  \end{equation}

  We proceed to show that each word $w$ in the
  list~\eqref{eq:small-typical} is typical. For that purpose, assume
  that $\varphi$ is an endomorphism of~$\{a,b\}^+$ such that
  $\varphi(w)$ is a segment of~$\mathbf{t}$.

  In the first three cases, since $\varphi(a^2)$ and $\varphi(b^2)$
  are factors of~$\varphi(w)$, we may start the argument using
  Proposition~\ref{p:Brlek} as in the proof of
  Theorem~\ref{t:substitutions}, assuming that $\varphi(a)$ is either
  $a$ or $aba$.

  Consider first the case $\varphi(a)=a$. Since in the three cases,
  $aab$ is a factor of~$w$ but $\mathbf{t}$ is cube-free, $\varphi(b)$
  must start with $b$. Therefore, we may assume that $\varphi(b)=bv$ with
  $v\in\{a,b\}^+$. In all three cases, since $abb$ is a
  factor of~$w$, we get that $abvbv$ is a factor of~$\varphi(w)$ and
  this would provide an overlap in~$\mathbf{t}$ if $v$~ends with~$a$.
  Hence $v$ ends with $b$. Since $\varphi(b)$ is of the form
  $\mu^n(x)$ for some $x\in\{a,aba,b,bab\}$, we conclude that either
  $\varphi(b)$ starts with $baab$ or it is $bab$. The first case is
  excluded since $aabaa$ is then a factor of~$\varphi(aab)$, whence
  also of~$\varphi(w)$, while it is not a segment of~$\mathbf{t}$. The
  case $\varphi(b)=bab$ is also excluded if $w$ is either $aababb$ or
  $aabbab$ since it leads to the overlap $babab$ in the factor
  $\varphi(bab)$ of~$\varphi(w)$. In case $w=abbaabba$, one can simply
  check directly that $\varphi(w)$ is not a segment of~$\mathbf{t}$.

  Still treating for the moment only the first three of the words in
  the list~\eqref{eq:small-typical}, suppose next that
  $\varphi(a)=aba$. Again, as $aab$ is a factor of~$w$ and $aabaa$
  cannot be a factor of~$\varphi(w)$, $\varphi(b)$ must start with
  $b$. If it ends with $a$, then $a\varphi(bb)$ would be an overlap
  in~$\varphi(w)$ since $abb$ is a factor of~$w$. Hence, $\varphi(b)$
  starts and ends with $b$. This is impossible in case $w$ has the
  factor $bab$ since it would lead to the overlap $babab$
  in~$\varphi(w)$. This excludes the cases where $w$ is the first or
  the second word in the list~\eqref{eq:small-typical}. So, we have
  $w=abbaabba$. Then $\varphi(w)$ is a square segment of~$\mathbf{t}$.
  By Proposition~\ref{p:Brlek}, $\varphi(abba)$ is one of the words
  $\mu^n(x)$ with $x\in\{a,aba,b,bab\}$. Since $n\le1$ gives a word
  that is too short to be $\varphi(abba)$, we must have $n\ge2$, in
  which case a simple calculation shows that $\mu^n(x)$ cannot start
  with $aba$. This ends the verification that the first three words in
  the list~\eqref{eq:small-typical} are typical.

  It remains to consider the word $w=ababba$. Here, we have two square
  factors of $\varphi(w)$, namely the squares of $\varphi(b)$ and
  $\varphi(ab)$. By Proposition~\ref{p:Brlek} we know that there are
  words $x,y\in\{a,aba,b,bab\}$ and non negative integers $m,n$ such
  that $\varphi(b)=\mu^m(x)$ and $\varphi(ab)=\mu^n(y)$. In case
  $m>n$, comparing the lengths of the word $\varphi(ab)$ and its
  factor $\varphi(b)$, we obtain the inequality $2^n|y|>2^m|x|$, so
  that $3\ge|y|>2^{m-n}|x|\ge2^{m-n}$. It follows that $|x|=1$,
  $|y|=3$ and $m=n+1$. From the equalities
  $\mu^n(y)=\varphi(ab)=\varphi(a)\mu^{n+1}(x)$ we then deduce that
  $\varphi(a)=\mu^n(y_1)$, where $y_1$ is the first letter of~$y$, and
  $\mu(x)=xy_1$. Since $abb$ is a factor of~$w$,
  $\varphi(abb)=\mu^n(y_1xy_1xy_1)$ is a segment of~$\mathbf{t}$,
  which contradicts $\mathbf{t}$ being overlap free. Thus, we must
  have $n\ge m$. Then $\varphi(a)$ must be of the form $\mu^m(z)$
  where $z$ is a prefix of~$\mu^{n-m}(y)$. It follows that, as in the
  proof of Theorem~\ref{t:substitutions}, we may then assume that
  $m=0$ and that $\varphi(b)$ is either $b$ or~$bab$. Consider first
  the case where $\varphi(b)=b$. Since $abba$ is a factor of~$w$ but
  $b^3$~is not a factor of $\varphi(w)$, the word $\varphi(a)$ must
  start and end with the letter~$a$ and we may assume that it is not
  reduced to~$a$. Since $\varphi(a)=z$, we conclude that $\varphi(a)$
  must start with $abba$. Since $bba$ is a factor of~$w$, this yields
  the factor $bbabb$ of~$\varphi(w)$, which is not possible since
  $\varphi(w)$ is a segment of~$\mathbf{t}$. Finally, the case
  $\varphi(b)=bab$ is excluded since $\varphi(ab)=\mu^n(y)$ cannot end
  with~$bab$. This concludes the proof of the theorem.
\end{proof}

To facilitate the visualization of the set of atypical words, we give
a semigroup theoretical formulation. Although we do not go deep into
it, the reader unfamiliar with semigroup theory may prefer to skip
these considerations or refer to a standard textbook in the area such
as~\citet{Clifford&Preston:1961,Howie:1995}.

Let $S$ be the set of atypical words. We may define a multiplication
on the set $S^0=S\cup\{0\}$ as follows: for $u,v\in S$, $u\cdot v$ is
$uv$ if $uv$ is atypical and 0 otherwise; for all $s\in S^0$,
$s\cdot0=0\cdot s=0$. Note that $S^0$ is the Rees quotient of
$\{a,b\}^+$ by the ideal consisting of the typical words together
with the words that are not segments of~$\mathbf{t}$.

The diagram in Figure~\ref{fig:eggbox} represents $S^0$ as a partially
ordered set for the Green $\mathcal{J}$-order, in which an element $u$
lies above $v$ if and only if $u$ is a factor of~$v$. The words in
bold are the lexicographic minima among their variants; note that
those that are atoms (which are underlined) are precisely the words that
were shown directly to be atypical in Theorem~\ref{t:atypical}.

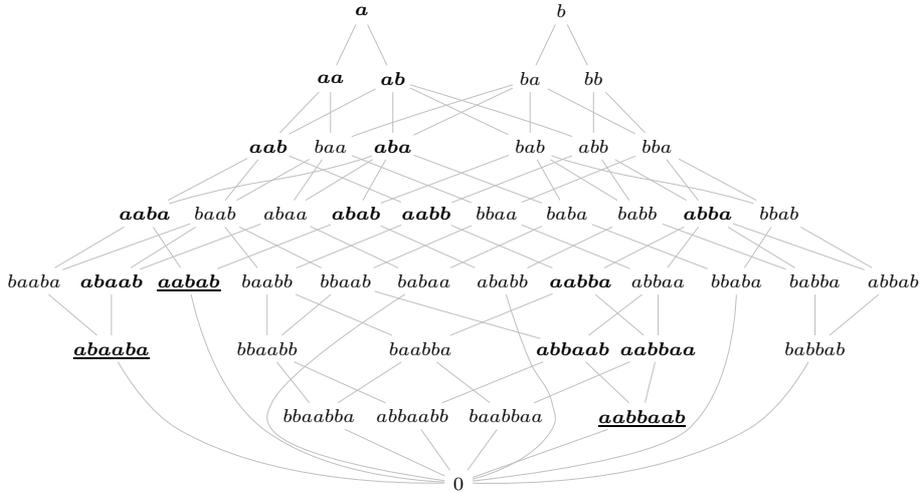
\begin{figure}[ht]
  \centering
  \tiny
  
\enlargethispage{5cm}
\begin{tikzpicture}[scale=.5,>=latex',line join=bevel]
  \pgfsetlinewidth{0.05mm}
\pgfsetcolor{gray}
  \draw [] (303.42bp,412.78bp) .. controls (298.59bp,402.59bp) and (291.38bp,387.35bp)  .. (286.56bp,377.17bp);
  \draw [] (313.58bp,412.78bp) .. controls (318.41bp,402.59bp) and (325.62bp,387.35bp)  .. (330.44bp,377.17bp);
  \draw [] (476.42bp,412.78bp) .. controls (471.59bp,402.59bp) and (464.38bp,387.35bp)  .. (459.56bp,377.17bp);
  \draw [] (486.77bp,412.78bp) .. controls (491.77bp,402.59bp) and (499.26bp,387.35bp)  .. (504.26bp,377.17bp);
  \draw [] (271.34bp,353.78bp) .. controls (261.69bp,343.59bp) and (247.25bp,328.35bp)  .. (237.61bp,318.17bp);
  \draw [] (281.5bp,353.78bp) .. controls (281.5bp,343.59bp) and (281.5bp,328.35bp)  .. (281.5bp,318.17bp);
  \draw [] (320.48bp,356.57bp) .. controls (300.89bp,346.24bp) and (266.65bp,328.16bp)  .. (245.63bp,317.07bp);
  \draw [] (335.5bp,353.78bp) .. controls (335.5bp,343.59bp) and (335.5bp,328.35bp)  .. (335.5bp,318.17bp);
  \draw [] (350.63bp,360.07bp) .. controls (377.4bp,352.14bp) and (434.78bp,334.81bp)  .. (482.5bp,318.0bp) .. controls (485.25bp,317.03bp) and (488.14bp,315.96bp)  .. (490.97bp,314.89bp);
  \draw [] (350.57bp,357.28bp) .. controls (372.21bp,346.92bp) and (412.07bp,327.83bp)  .. (435.63bp,316.54bp);
  \draw [] (439.43bp,357.28bp) .. controls (417.54bp,346.8bp) and (377.01bp,327.38bp)  .. (353.56bp,316.15bp);
  \draw [] (439.47bp,360.05bp) .. controls (412.89bp,352.11bp) and (355.91bp,334.75bp)  .. (308.5bp,318.0bp) .. controls (305.65bp,316.99bp) and (302.66bp,315.88bp)  .. (299.74bp,314.77bp);
  \draw [] (454.5bp,353.78bp) .. controls (454.5bp,343.59bp) and (454.5bp,328.35bp)  .. (454.5bp,318.17bp);
  \draw [] (469.8bp,356.57bp) .. controls (489.63bp,346.3bp) and (524.22bp,328.37bp)  .. (545.65bp,317.27bp);
  \draw [] (509.5bp,353.78bp) .. controls (509.5bp,343.59bp) and (509.5bp,328.35bp)  .. (509.5bp,318.17bp);
  \draw [] (519.85bp,353.78bp) .. controls (529.68bp,343.59bp) and (544.38bp,328.35bp)  .. (554.2bp,318.17bp);
  \draw [] (209.35bp,295.83bp) .. controls (190.27bp,285.67bp) and (160.33bp,269.72bp)  .. (140.56bp,259.19bp);
  \draw [] (245.69bp,297.93bp) .. controls (270.73bp,287.51bp) and (315.76bp,268.78bp)  .. (342.61bp,257.61bp);
  \draw [] (218.85bp,294.78bp) .. controls (210.62bp,284.59bp) and (198.32bp,269.35bp)  .. (190.11bp,259.17bp);
  \draw [] (317.33bp,298.02bp) .. controls (314.4bp,296.93bp) and (311.39bp,295.88bp)  .. (308.5bp,295.0bp) .. controls (239.64bp,273.89bp) and (219.71bp,278.93bp)  .. (150.5bp,259.0bp) .. controls (147.64bp,258.18bp) and (144.65bp,257.23bp)  .. (141.71bp,256.26bp);
  \draw [] (318.01bp,294.78bp) .. controls (301.38bp,284.59bp) and (276.52bp,269.35bp)  .. (259.91bp,259.17bp);
  \draw [] (329.48bp,294.78bp) .. controls (323.76bp,284.59bp) and (315.2bp,269.35bp)  .. (309.49bp,259.17bp);
  \draw [] (353.71bp,298.62bp) .. controls (381.64bp,288.08bp) and (434.93bp,267.97bp)  .. (464.81bp,256.69bp);
  \draw [] (490.84bp,298.16bp) .. controls (464.07bp,287.64bp) and (414.86bp,268.3bp)  .. (386.4bp,257.11bp);
  \draw [] (528.12bp,294.78bp) .. controls (545.82bp,284.59bp) and (572.29bp,269.35bp)  .. (589.96bp,259.17bp);
  \draw [] (516.65bp,294.78bp) .. controls (523.44bp,284.59bp) and (533.6bp,269.35bp)  .. (540.38bp,259.17bp);
  \draw [] (274.16bp,294.78bp) .. controls (267.19bp,284.59bp) and (256.76bp,269.35bp)  .. (249.8bp,259.17bp);
  \draw [] (263.16bp,295.05bp) .. controls (245.23bp,284.82bp) and (218.07bp,269.35bp)  .. (200.06bp,259.08bp);
  \draw [] (299.74bp,298.28bp) .. controls (326.23bp,287.8bp) and (375.27bp,268.38bp)  .. (403.65bp,257.15bp);
  \draw [] (435.98bp,298.51bp) .. controls (407.95bp,287.93bp) and (354.85bp,267.88bp)  .. (325.09bp,256.65bp);
  \draw [] (460.52bp,294.78bp) .. controls (466.24bp,284.59bp) and (474.8bp,269.35bp)  .. (480.51bp,259.17bp);
  \draw [] (471.99bp,294.78bp) .. controls (488.62bp,284.59bp) and (513.48bp,269.35bp)  .. (530.09bp,259.17bp);
  \draw [] (473.06bp,298.15bp) .. controls (476.19bp,297.02bp) and (479.42bp,295.93bp)  .. (482.5bp,295.0bp) .. controls (551.06bp,274.39bp) and (570.62bp,278.51bp)  .. (639.5bp,259.0bp) .. controls (642.34bp,258.19bp) and (645.31bp,257.28bp)  .. (648.23bp,256.34bp);
  \draw [] (572.78bp,294.78bp) .. controls (580.64bp,284.59bp) and (592.41bp,269.35bp)  .. (600.26bp,259.17bp);
  \draw [] (545.75bp,297.81bp) .. controls (520.11bp,287.3bp) and (474.21bp,268.47bp)  .. (447.14bp,257.37bp);
  \draw [] (583.45bp,295.31bp) .. controls (602.39bp,285.13bp) and (631.38bp,269.54bp)  .. (650.63bp,259.18bp);
  \draw [] (127.65bp,235.78bp) .. controls (134.44bp,225.59bp) and (144.6bp,210.35bp)  .. (151.38bp,200.17bp);
  \draw [] (102.44bp,235.78bp) .. controls (85.28bp,225.59bp) and (59.612bp,210.35bp)  .. (42.476bp,200.17bp);
  \draw [] (386.03bp,237.34bp) .. controls (410.0bp,227.15bp) and (448.67bp,210.7bp)  .. (473.84bp,199.99bp);
  \draw [] (342.94bp,237.59bp) .. controls (318.22bp,227.39bp) and (277.78bp,210.68bp)  .. (251.65bp,199.89bp);
  \draw [] (221.19bp,238.45bp) .. controls (193.55bp,228.02bp) and (145.43bp,209.86bp)  .. (116.22bp,198.83bp);
  \draw [] (263.68bp,236.44bp) .. controls (285.21bp,226.21bp) and (318.43bp,210.43bp)  .. (340.32bp,200.04bp);
  \draw [] (281.79bp,237.97bp) .. controls (255.4bp,227.59bp) and (210.98bp,210.13bp)  .. (183.28bp,199.24bp);
  \draw [] (325.04bp,236.83bp) .. controls (347.69bp,226.67bp) and (383.22bp,210.72bp)  .. (406.69bp,200.19bp);
  \draw [] (588.07bp,235.91bp) .. controls (568.31bp,225.67bp) and (538.54bp,210.25bp)  .. (518.82bp,200.03bp);
  \draw [] (600.41bp,235.78bp) .. controls (592.72bp,225.59bp) and (581.23bp,210.35bp)  .. (573.55bp,200.17bp);
  \draw [] (630.21bp,238.81bp) .. controls (659.91bp,228.3bp) and (713.08bp,209.47bp)  .. (744.43bp,198.37bp);
  \draw [] (625.99bp,235.78bp) .. controls (642.62bp,225.59bp) and (667.48bp,210.35bp)  .. (684.09bp,200.17bp);
  \draw [] (164.57bp,235.78bp) .. controls (148.48bp,225.59bp) and (124.42bp,210.35bp)  .. (108.35bp,200.17bp);
  \draw [] (160.0bp,238.69bp) .. controls (131.03bp,228.17bp) and (79.576bp,209.5bp)  .. (49.095bp,198.43bp);
  \draw [] (189.96bp,235.78bp) .. controls (198.01bp,225.59bp) and (210.04bp,210.35bp)  .. (218.07bp,200.17bp);
  \draw [] (202.49bp,235.91bp) .. controls (222.79bp,225.67bp) and (253.36bp,210.25bp)  .. (273.62bp,200.03bp);
  \draw [] (464.9bp,236.57bp) .. controls (442.7bp,226.36bp) and (408.25bp,210.53bp)  .. (385.53bp,200.09bp);
  \draw [] (508.19bp,238.09bp) .. controls (534.87bp,227.75bp) and (580.06bp,210.22bp)  .. (608.27bp,199.28bp);
  \draw [] (525.77bp,235.91bp) .. controls (504.75bp,225.67bp) and (473.09bp,210.25bp)  .. (452.12bp,200.03bp);
  \draw [] (569.24bp,238.45bp) .. controls (597.43bp,228.02bp) and (646.5bp,209.86bp)  .. (676.29bp,198.83bp);
  \draw [] (447.07bp,237.72bp) .. controls (472.38bp,227.41bp) and (514.26bp,210.36bp)  .. (540.84bp,199.54bp);
  \draw [] (403.87bp,237.09bp) .. controls (380.47bp,226.91bp) and (343.24bp,210.71bp)  .. (318.84bp,200.09bp);
  \draw [] (689.12bp,235.78bp) .. controls (706.82bp,225.59bp) and (733.29bp,210.35bp)  .. (750.96bp,200.17bp);
  \draw [] (663.54bp,235.78bp) .. controls (656.93bp,225.59bp) and (647.03bp,210.35bp)  .. (640.43bp,200.17bp);
  \draw [] (160.6bp,176.85bp) .. controls (166.09bp,151.93bp) and (183.36bp,89.641bp)  .. (223.5bp,59.0bp) .. controls (272.6bp,21.52bp) and (350.1bp,14.242bp)  .. (380.08bp,12.835bp);
  \draw [] (511.1bp,176.78bp) .. controls (523.08bp,166.59bp) and (540.99bp,151.35bp)  .. (552.95bp,141.17bp);
  \draw [] (473.97bp,177.44bp) .. controls (449.03bp,167.21bp) and (410.55bp,151.43bp)  .. (385.2bp,141.04bp);
  \draw [] (91.5bp,176.78bp) .. controls (91.5bp,166.59bp) and (91.5bp,151.35bp)  .. (91.5bp,141.17bp);
  \draw [] (433.59bp,176.95bp) .. controls (437.75bp,163.16bp) and (445.67bp,138.38bp)  .. (454.5bp,118.0bp) .. controls (461.7bp,101.4bp) and (469.15bp,99.561bp)  .. (473.5bp,82.0bp) .. controls (475.96bp,72.078bp) and (478.62bp,67.846bp)  .. (473.5bp,59.0bp) .. controls (458.55bp,33.18bp) and (423.25bp,20.454bp)  .. (404.57bp,15.338bp);
  \draw [] (565.5bp,176.78bp) .. controls (565.5bp,166.59bp) and (565.5bp,151.35bp)  .. (565.5bp,141.17bp);
  \draw [] (551.58bp,176.78bp) .. controls (538.35bp,166.59bp) and (518.57bp,151.35bp)  .. (505.36bp,141.17bp);
  \draw [] (756.71bp,176.78bp) .. controls (744.55bp,166.59bp) and (726.37bp,151.35bp)  .. (714.23bp,141.17bp);
  \draw [] (37.102bp,176.78bp) .. controls (49.081bp,166.59bp) and (66.994bp,151.35bp)  .. (78.954bp,141.17bp);
  \draw [] (226.5bp,176.78bp) .. controls (226.5bp,166.59bp) and (226.5bp,151.35bp)  .. (226.5bp,141.17bp);
  \draw [] (251.2bp,176.91bp) .. controls (275.09bp,166.67bp) and (311.08bp,151.25bp)  .. (334.93bp,141.03bp);
  \draw [] (345.95bp,176.94bp) .. controls (313.36bp,155.74bp) and (242.93bp,107.84bp)  .. (230.5bp,82.0bp) .. controls (226.07bp,72.788bp) and (224.35bp,67.165bp)  .. (230.5bp,59.0bp) .. controls (248.99bp,34.46bp) and (346.16bp,18.891bp)  .. (380.48bp,14.088bp);
  \draw [] (701.5bp,176.78bp) .. controls (701.5bp,166.59bp) and (701.5bp,151.35bp)  .. (701.5bp,141.17bp);
  \draw [] (281.71bp,176.78bp) .. controls (269.55bp,166.59bp) and (251.37bp,151.35bp)  .. (239.23bp,141.17bp);
  \draw [] (319.73bp,179.64bp) .. controls (322.69bp,178.74bp) and (325.66bp,177.84bp)  .. (328.5bp,177.0bp) .. controls (375.75bp,162.97bp) and (430.95bp,147.39bp)  .. (463.47bp,138.3bp);
  \draw [] (633.33bp,176.67bp) .. controls (632.33bp,151.92bp) and (626.27bp,90.887bp)  .. (591.5bp,59.0bp) .. controls (563.52bp,33.34bp) and (443.26bp,18.086bp)  .. (404.55bp,13.772bp);
  \draw [] (97.206bp,117.82bp) .. controls (106.14bp,102.2bp) and (124.99bp,73.434bp)  .. (149.5bp,59.0bp) .. controls (227.97bp,12.792bp) and (343.12bp,11.227bp)  .. (380.39bp,12.054bp);
  \draw [] (695.8bp,117.8bp) .. controls (686.88bp,102.16bp) and (668.05bp,73.362bp)  .. (643.5bp,59.0bp) .. controls (562.21bp,11.451bp) and (442.93bp,10.801bp)  .. (404.74bp,11.953bp);
  \draw [] (562.87bp,117.78bp) .. controls (560.36bp,107.59bp) and (556.62bp,92.348bp)  .. (554.12bp,82.173bp);
  \draw [] (540.98bp,117.91bp) .. controls (517.27bp,107.67bp) and (481.55bp,92.25bp)  .. (457.89bp,82.031bp);
  \draw [] (249.9bp,117.91bp) .. controls (272.54bp,107.67bp) and (306.63bp,92.25bp)  .. (329.22bp,82.031bp);
  \draw [] (234.96bp,117.78bp) .. controls (243.01bp,107.59bp) and (255.04bp,92.348bp)  .. (263.07bp,82.173bp);
  \draw [] (502.79bp,117.78bp) .. controls (513.51bp,107.59bp) and (529.55bp,92.348bp)  .. (540.27bp,82.173bp);
  \draw [] (465.68bp,117.91bp) .. controls (440.71bp,107.67bp) and (403.1bp,92.25bp)  .. (378.18bp,82.031bp);
  \draw [] (373.42bp,117.78bp) .. controls (386.65bp,107.59bp) and (406.43bp,92.348bp)  .. (419.64bp,82.173bp);
  \draw [] (342.95bp,117.78bp) .. controls (327.21bp,107.59bp) and (303.69bp,92.348bp)  .. (287.98bp,82.173bp);
  \draw [] (521.97bp,58.912bp) .. controls (486.76bp,46.292bp) and (429.57bp,25.79bp)  .. (404.7bp,16.874bp);
  \draw [] (360.02bp,58.779bp) .. controls (367.18bp,48.588bp) and (377.87bp,33.348bp)  .. (385.01bp,23.173bp);
  \draw [] (425.79bp,58.779bp) .. controls (418.46bp,48.588bp) and (407.5bp,33.348bp)  .. (400.18bp,23.173bp);
  \draw [] (293.98bp,58.912bp) .. controls (319.54bp,46.871bp) and (360.33bp,27.657bp)  .. (380.46bp,18.173bp);
  \draw (308.5bp,424.5bp) node {$\bm{a}$};
  \draw (281.5bp,365.5bp) node {$\bm{aa}$};
  \draw (335.5bp,365.5bp) node {$\bm{ab}$};
  \draw (481.5bp,424.5bp) node {$b$};
  \draw (454.5bp,365.5bp) node {$ba$};
  \draw (509.5bp,365.5bp) node {$bb$};
  \draw (227.5bp,306.5bp) node {$\bm{aab}$};
  \draw (281.5bp,306.5bp) node {$baa$};
  \draw (335.5bp,306.5bp) node {$\bm{aba}$};
  \draw (509.5bp,306.5bp) node {$abb$};
  \draw (454.5bp,306.5bp) node {$bab$};
  \draw (564.5bp,306.5bp) node {$bba$};
  \draw (120.5bp,247.5bp) node {$\bm{aaba}$};
  \draw (364.5bp,247.5bp) node {$\bm{aabb}$};
  \draw (181.5bp,247.5bp) node {$baab$};
  \draw (242.5bp,247.5bp) node {$abaa$};
  \draw (303.5bp,247.5bp) node {$\bm{abab}$};
  \draw (486.5bp,247.5bp) node {$baba$};
  \draw (608.5bp,247.5bp) node {$\bm{abba}$};
  \draw (547.5bp,247.5bp) node {$babb$};
  \draw (425.5bp,247.5bp) node {$bbaa$};
  \draw (670.5bp,247.5bp) node {$bbab$};
  \draw (158.5bp,186.5bp) node {$\underline{\bm{aabab}}$};
  \draw (24.5bp,188.5bp) node {$baaba$};
  \draw (498.5bp,188.5bp) node {$\bm{aabba}$};
  \draw (226.5bp,188.5bp) node {$baabb$};
  \draw (91.5bp,188.5bp) node {$\bm{abaab}$};
  \draw (362.5bp,188.5bp) node {$babaa$};
  \draw (430.5bp,188.5bp) node {$ababb$};
  \draw (565.5bp,188.5bp) node {$abbaa$};
  \draw (769.5bp,188.5bp) node {$abbab$};
  \draw (701.5bp,188.5bp) node {$babba$};
  \draw (294.5bp,188.5bp) node {$bbaab$};
  \draw (633.5bp,188.5bp) node {$bbaba$};
  \draw (392.5bp,11.5bp) node {$0$};
  \draw (565.5bp,129.5bp) node {$\bm{aabbaa}$};
  \draw (359.5bp,129.5bp) node {$baabba$};
  \draw (91.5bp,127.5bp) node {$\underline{\bm{abaaba}}$};
  \draw (491.5bp,129.5bp) node {$\bm{abbaab}$};
  \draw (701.5bp,129.5bp) node {$babbab$};
  \draw (226.5bp,129.5bp) node {$bbaabb$};
  \draw (551.5bp,68.5bp) node {$\underline{\bm{aabbaab}}$};
  \draw (433.5bp,70.5bp) node {$baabbaa$};
  \draw (352.5bp,70.5bp) node {$abbaabb$};
  \draw (271.5bp,70.5bp) node {$bbaabba$};
\end{tikzpicture}

  \caption{The semigroup $S^0$}
  \label{fig:eggbox}
\end{figure}

We conclude this section with another application of
Theorem~\ref{t:substitutions}, this one concerning infinite patterns
of~$\mathbf{t}$.

\begin{Cor}
  \label{c:substitutions}
  Let $w$ be an infinite word and suppose that there is an endomorphism
  $\varphi$ of~$\{a,b\}^+$ such that $\varphi(w)$ is a suffix of
  either $\mathbf{t}$ or $\xi(\mathbf{t})$. Then $w$ is itself a
  suffix of either $\mathbf{t}$ or $\xi(\mathbf{t})$.
\end{Cor}

\begin{proof}
  Since all segments of $w$ are unavoidable in~$\mathbf{t}$ and they
  are all extendable on the right, by
  Theorem~\ref{t:unavoidable-in-TM} they are segments of~$\mathbf{t}$.
  Since the language of the segments of~$\mathbf{t}$ defines a minimal
  subshift \citep[Proposition~5.2]{Queffelec:2010}, it follows that $w$
  and~$\mathbf{t}$ have the same segments. In particular, the word
  $a^2b^2a^2bab$ is a segment of~$w$ and it satisfies the assumption
  of Theorem~\ref{t:substitutions}. It follows that there is $n\ge0$
  such that $\varphi=\mu^n$ or $\varphi=\xi\circ\mu^n$. Again, since
  $\mu$ is injective and both $\mathbf{t}$ and $\xi(\mathbf{t})$ are
  fixed by~$\mu$, the result follows.
\end{proof}

The somewhat different formulation for finite and infinite segments
(compare Theorem~\ref{t:unavoidable-in-TM} with
Corollary~\ref{c:substitutions}) is fully justified by the following
result, which entails that the infinite words $\mathbf{t}$
and~$\xi(\mathbf{t})$ have no common suffix.

\begin{Prop}
  \label{p:suffixes}
  If $\mathbf{s}$ is an infinite word over $\{a,b\}$ and $\mathbf{w}$
  is a common infinite suffix of $\mathbf{s}$ and $\xi(\mathbf{s})$,
  then $\mathbf{w}$ is periodic.
\end{Prop}

\begin{proof}
  By assumption, there are finite words $x$ and $y$ such that
  $\mathbf{s}=x\mathbf{w}$, $\xi(\mathbf{s})=y\mathbf{w}$. Since
  $\mathbf{w}$ and $\xi(\mathbf{w})$ start with different letters, the
  words $x$ and $y$ have different lengths. Replacing $\mathbf{s}$ by
  $\xi(\mathbf{s})$, if needed, we may assume that $x$ is shorter than
  $y$. As $\xi(x)$ is a prefix of $\xi(\mathbf{s})=y\mathbf{w}$, it
  follows that $y=\xi(x)z$ for some word $z$. From
  $\xi(\mathbf{s})=\xi(x)z\mathbf{w}$, we deduce that
  $\xi(\mathbf{w})=z\mathbf{w}$ and so
  $\mathbf{w}=\xi^2(\mathbf{w})=\xi(z)z\mathbf{w}$, thereby showing that
  $\mathbf{w}$ is periodic.
\end{proof}

\section{Final remarks and problems}
\label{sec:problems}

For an infinite word $\mathbf{w}$ over a finite alphabet $A$, let
$L(\mathbf{w})$ be the language consisting of its finite segments.
Note that the automorphisms of the semigroup $A^+$ permute the letters
of~$A$; we call them \emph{letter exchanges}. The language obtained
from $L(\mathbf{w})$ by applying all possible letter exchanges is
denoted $\bar L(\mathbf{w})$. Let $\Cl E(\mathbf{w})$ denote the set
of all endomorphisms $\varphi$ of~$A^+$ such that
$\varphi(L(\mathbf{w}))\subseteq L(\mathbf{w})$. The set $\bar{\Cl
  E}(\mathbf{w})$ is similarly defined using $\bar L(\mathbf{w})$
instead of~$L(\mathbf{w})$. Note that both $\Cl E(\mathbf{w})$ and
$\bar{\Cl E}(\mathbf{w})$ are submonoids of the monoid $\End(A^+)$ of
all endomorphisms of the semigroup~$A^+$.

The following is an immediate consequence of
Theorem~\ref{t:substitutions}.

\begin{Cor}
  \label{c:Et}
  The monoid $\Cl E(\mathbf{t})=\bar{\Cl E}(\mathbf{t})$ is generated
  by the set $\{\xi,\mu\}$. In particular, it is finitely
  generated.\qed
\end{Cor}

Corollary~\ref{c:Et} is intimately related with a result of Thue
(see~\citet[Chapter~3, Theorem~2.16]{Berstel:1995}) that characterizes
the set of the so-called \emph{overlap-free morphisms}, that is,
endomorphisms of $\{a,b\}^+$ that map the set of all overlap-free
words into itself, namely as the monoid generated by $\{\xi,\mu\}$. In
fact, in view of another result of Thue (see~\citet[Chapter~3,
Theorem~2.15]{Berstel:1995}), all (overlap-free) words that can be
arbitrarily prolonged in both directions to overlap-free words are
segments of~$\mathbf{t}$. It follows that overlap-free morphisms
belong to $\Cl E(\mathbf{t})$ and so Corollary~\ref{c:Et} immediately
yields Thue's necessary condition for overlap-free morphisms. That the
condition is also sufficient is given by another result of Thue
(see~\citet[Chapter~3, Lemma~2.2]{Berstel:1995}). It does not appear to
be immediately obvious how to deduce Corollary~\ref{c:Et} from Thue's
results.

Corollary~\ref{c:Et} is also related with a result of
\citet{Pansiot:1981} characterizing the endomorphisms of
$\{a,b\}$ that generate some infinite word obtained from $\mathbf{t}$
by dropping a finite prefix as precisely the powers of~$\mu$. Since
$\mathbf{t}$ is recurrent, all such infinite words $\mathbf{w}$ have
the same language $L(\mathbf{w})=L(\mathbf{t})$. Hence, the
endomorphisms $\varphi$ considered by Pansiot belong to $\Cl
E(\mathbf{t})$, whence they are products of $\xi$ and $\mu$. Since
$\xi$ and $\mu$ commute, it follows from Corollary~\ref{c:Et} that
$\varphi$ is either $\mu^k$ or $\xi\mu^k$ for some $k\ge0$, the latter
possibility being excluded because $\mathbf{w}$ is assumed to be a
fixed point of~$\varphi$. This gives Pansiot's result. Again, it is
not clear how to deduce Corollary~\ref{c:Et} from Pansiot's results.

Theorems~\ref{t:unavoidable-in-TM} and~\ref{t:substitutions}, together
with Corollary~\ref{c:Et} may be regarded as three finiteness
properties of the Prouhet-Thue-Morse sequence. It is natural to ask
which infinite words possess such finiteness properties. More
precisely, we propose the following problems.

\begin{Problem}
  \label{pb:1}
  Which infinite words $\mathbf{w}$ have the property that, up to
  finitely many exceptions, the patterns of~$\mathbf{w}$ on the same
  alphabet are obtained from its segments up to an exchange of
  letters?
\end{Problem}

\begin{Problem}
  \label{pb:2}
  For which infinite words $\mathbf{w}$ is the monoid $\Cl
  E(\mathbf{w})$ finitely generated? Similar question for $\bar{\Cl
    E}(\mathbf{w})$.
\end{Problem}

We say that a finite segment $u$ of~$\mathbf{w}$ is
\emph{$\mathbf{w}$-atypical} if there is some endomorphism
$\varphi\notin\bar{\Cl E}(\mathbf{w})$ of~$A^+$ such that $\varphi(u)$
is also a segment of~$\mathbf{w}$.

\begin{Problem}
  \label{pb:3}
  Which infinite words $\mathbf{w}$ have only finitely many
  $\mathbf{w}$-atypical segments?
\end{Problem}

A negative example for Problem~\ref{pb:1} is provided by the
\emph{Fibonacci infinite word}, which is the only fixed point
$\mathbf{f}$ of the endomorphism \emph{$\phi$} of~$\{a,b\}^+$ defined
by $\phi(a)=ab$ and $\phi(b)=a$. That there are infinitely many finite
binary patterns of~$\mathbf{f}$ that are not segments of~$\mathbf{f}$
was proved in~\citet{Restivo&Salemi:2002a} (see
also~\citet{Restivo&Salemi:2002b}), where it is also shown that there
are Sturmian infinite words that admit as patterns all segments of all
Sturmian infinite words. Recall that an infinite word is
\emph{Sturmian} if it has exactly $n+1$ segments of each length
$n\ge1$. We do not know whether $\Cl E(\mathbf{f})$ is generated
by~$\varphi$ and $\bar{\Cl E}(\mathbf{f})$ is generated by $\varphi$
and $\xi$. We also do not know whether the set of
$\mathbf{f}$-atypical words is finite.

Problem~\ref{pb:1} was raised in~\citet{Restivo&Salemi:2002a} for
binary infinite words that are either fixed points of endomorphisms or
of linear complexity. In the same paper, it is observed that if
$\mathbf{w}$ is an infinite word with all elements of $A^+$ as
segments (which may be obtained for instance by concatenating all the
words in a sequence enumerating the elements of~$A^+$), then obviously
$\mathbf{w}$ is a positive example for Problem~\ref{pb:1}. Note that
$\Cl E(\mathbf{w})=\bar{\Cl E}(\mathbf{w})=\End(A^+)$ and it is easy
to see that $\End(A^+)$ is not finitely generated: for the
endomorphisms that maps each letter to itself, except for one letter
$a$ that is mapped to~$a^p$, where $p$ is prime, the only elements
of~$\End(\mathbf{w})$ that are factors of it are the letter exchanges
and the factors of which it is also a factor. From the preceding
observation it also follows that there are no $\mathbf{w}$-atypical
words. Thus, $\mathbf{w}$ is a negative example for Problem~\ref{pb:2}
and a positive example for Problem~\ref{pb:3}.

\acknowledgments{
Besides the connection with Pansiot's work, we thank the anonymous
referee for comments that led to improved readability of this paper.}

\bibliographystyle{abbrvnat}
\bibliography{sgpabb,ref-sgps}

\begin{thebibliography}{22}
\providecommand{\natexlab}[1]{#1}
\providecommand{\url}[1]{\texttt{#1}}
\expandafter\ifx\csname urlstyle\endcsname\relax
  \providecommand{\doi}[1]{doi: #1}\else
  \providecommand{\doi}{doi: \begingroup \urlstyle{rm}\Url}\fi

\bibitem[Allouche and Shallit(1999)]{Allouche&Shallit:1999}
J.-P. Allouche and J.~Shallit.
\newblock The ubiquitous {P}rouhet-{T}hue-{M}orse sequence.
\newblock In \emph{Sequences and their applications ({S}ingapore, 1998)},
  Springer Ser. Discrete Math. Theor. Comput. Sci., pages 1--16. Springer,
  London, 1999.

\bibitem[Allouche and Shallit(2003)]{Allouche&Shallit:2003}
J.-P. Allouche and J.~Shallit.
\newblock \emph{Automatic Sequences: Theory, Applications, Generalizations}.
\newblock Cambridge University Press, 2003.

\bibitem[Almeida and Costa(2013)]{Almeida&ACosta:2013}
J.~Almeida and A.~Costa.
\newblock Presentations of {S}ch\"utzenberger groups of minimal subshifts.
\newblock \emph{Israel J. Math.}, 196:\penalty0 1--31, 2013.

\bibitem[Almeida et~al.(2020)Almeida, Costa, Kyriakoglou, and
  Perrin]{Almeida&ACosta&Kyriakoglou&Perrin:2020b}
J.~Almeida, A.~Costa, R.~Kyriakoglou, and D.~Perrin.
\newblock \emph{Profinite semigroups and symbolic dynamics}, volume 2274 of
  \emph{Lect. Notes in Math.}
\newblock Springer, Cham, 2020.

\bibitem[Berstel(1995)]{Berstel:1995}
J.~Berstel.
\newblock Axel {T}hue's papers on repetitions in words: a translation.
\newblock Technical report, Université du Québec à Montréal, 1995.
\newblock Publications du LaCIM 20, available at
  \url{http://www-igm.univ-mlv.fr/~berstel/Articles/1994ThueTranslation.pdf}.

\bibitem[Brlek(1989)]{Brlek:1989a}
S.~Brlek.
\newblock Enumeration of factors in the {T}hue-{M}orse word.
\newblock \emph{Discrete Appl. Math.}, 24:\penalty0 83--96, 1989.

\bibitem[Clifford and Preston(1961)]{Clifford&Preston:1961}
A.~H. Clifford and G.~B. Preston.
\newblock \emph{The Algebraic Theory of Semigroups}, volume~I.
\newblock Amer. Math. Soc., Providence, R.I., 1961.

\bibitem[de~Luca and Varricchio(1989)]{Luca&Varricchio:1989}
A.~de~Luca and S.~Varricchio.
\newblock Some combinatorial properties of the {T}hue-{M}orse sequence and a
  problem in semigroups.
\newblock \emph{Theor. Comp. Sci.}, 63\penalty0 (3):\penalty0 333--348, 1989.

\bibitem[Howie(1995)]{Howie:1995}
J.~M. Howie.
\newblock \emph{Fundamentals of semigroup theory}, volume~12 of \emph{London
  Mathematical Society Monographs. New Series}.
\newblock The Clarendon Press, Oxford University Press, New York, 1995.

\bibitem[Lothaire(1983)]{Lothaire:1983}
M.~Lothaire.
\newblock \emph{Combinatorics on Words}.
\newblock Addison-Wesley, Reading, Mass., 1983.

\bibitem[Morse(1921)]{Morse:1921a}
H.~M. Morse.
\newblock Recurrent geodesics on a surface of negative curvature.
\newblock \emph{Trans. Amer. Math. Soc.}, 22:\penalty0 84--100, 1921.

\bibitem[Pansiot(1981)]{Pansiot:1981}
J.-J. Pansiot.
\newblock The {M}orse sequence and iterated morphisms.
\newblock \emph{Inform. Process. Lett.}, 12\penalty0 (2):\penalty0 68--70,
  1981.

\bibitem[Prouhet(1851)]{Prouhet:1851}
E.~Prouhet.
\newblock Mémoire sur quelques relations entre les puissances des nombres.
\newblock \emph{C. R. Acad. Sci. Paris}, 33:\penalty0 31, 1851.

\bibitem[Queff\'{e}lec(2010)]{Queffelec:2010}
M.~Queff\'{e}lec.
\newblock \emph{Substitution dynamical systems---spectral analysis}, volume
  1294 of \emph{Lect. Notes in Math.}
\newblock Springer-Verlag, Berlin, second edition, 2010.

\bibitem[Restivo and Salemi(2002{\natexlab{a}})]{Restivo&Salemi:2002a}
A.~Restivo and S.~Salemi.
\newblock Words and patterns.
\newblock In \emph{Developments in language theory ({V}ienna, 2001)}, volume
  2295 of \emph{Lect. Notes in Comput. Sci.}, pages 117--129. Springer, Berlin,
  2002{\natexlab{a}}.

\bibitem[Restivo and Salemi(2002{\natexlab{b}})]{Restivo&Salemi:2002b}
A.~Restivo and S.~Salemi.
\newblock Binary patterns in infinite binary words.
\newblock In \emph{Formal and natural computing}, volume 2300 of \emph{Lect.
  Notes in Comput. Sci.}, pages 107--116. Springer, Berlin, 2002{\natexlab{b}}.

\bibitem[Shur(1996{\natexlab{a}})]{Shur:1996a}
A.~M. Shur.
\newblock Binary words avoided by the {T}hue-{M}orse sequence.
\newblock \emph{Semigroup Forum}, 53:\penalty0 212--219, 1996{\natexlab{a}}.

\bibitem[Shur(1996{\natexlab{b}})]{Shur:1996b}
A.~M. Shur.
\newblock Overlap-free words and {T}hue-{M}orse sequences.
\newblock \emph{Int. J. Algebra Comput.}, 6:\penalty0 353--367,
  1996{\natexlab{b}}.

\bibitem[Shur(1997)]{Shur:1997PhD}
A.~M. Shur.
\newblock \emph{Algebraic and combinatorial properties of rational languages}.
\newblock PhD thesis, Ural State University, 1997.
\newblock In Russian.

\bibitem[Shur(2005)]{Shur:2005}
A.~M. Shur.
\newblock Combinatorial complexity of rational languages.
\newblock \emph{Diskretn. Anal. Issled. Oper.}, 12\penalty0 (2):\penalty0
  78--99, 2005.
\newblock In Russian.

\bibitem[Thue(1906)]{Thue:1906}
A.~Thue.
\newblock Über unendlichen {Z}eichenreihen.
\newblock \emph{Kra. Vidensk. Selsk. Skrifter, I. Mat. Nat. Kl.}, \penalty0
  (7):\penalty0 1--22, 1906.

\bibitem[Thue(1912)]{Thue:1912bis}
A.~Thue.
\newblock Über die gegenseitige {L}age gleicher {T}eile gewisser
  {Z}eichenreihen.
\newblock \emph{Kra. Vidensk. Selsk. Skrifter, I. Mat. Nat. Kl.}, \penalty0
  (1):\penalty0 1--67, 1912.

\end{thebibliography}
\end{document}